\documentclass{amsart}
\usepackage{amsmath}%
\usepackage{amsfonts}%
\usepackage{amssymb}%
\usepackage{graphicx}
\newtheorem{theorem}{Theorem}
\theoremstyle{plain}

\newtheorem{conjecture}{Conjecture}

\newtheorem{lemma}{Lemma}

\newtheorem{proposition}{Proposition}
\newtheorem{remark}{Remark}

\numberwithin{equation}{section}

\begin{document}

\title{On extremums of sums of powered distances to a finite set of points}

\author{Nikolai Nikolov and Rafael Rafailov}
\address
{Institute of Mathematics and Informatics,
Bulgarian Academy of Sciences,
Acad. G. Bonchev 8, 1113 Sofia,
Bulgaria}

\email{nik@math.bas.bg}
\address
{University of California, Berkeley}
\email{rafael.rafailov@yahoo.com}



\dedicatory{}

\begin{abstract}
In this paper we investigate the extremal properties of the sum
$$\sum_{i=1}^n|MA_i|^{\lambda},$$
where $A_i$ are vertices of a regular simplex, a cross-polytope
(orthoplex) or a cube and $M$ varies on a sphere concentric to the
sphere circumscribed around one of the given polytopes. We give
full characterization for which points on $\Gamma$ the extremal
values of the sum are obtained in terms of $\lambda$. In the case
of the regular dodecahedron and icosahedron in $\mathbb{R}^3$ we
obtain results for which values of $\lambda$ the corresponding sum
is independent of the position of $M$ on $\Gamma$. We use
elementary analytic and purely geometric methods.
\end{abstract}
\maketitle
\section{Introduction}

We investigate problems of the extremality of sums of the kind
\begin{equation}\label{OsSu}\sum_{i=1}^n|MA_i|^{\lambda},\end{equation}
where $A_i$ are vertices of a regular simplex, a cross-polytope
(orthoplex) or a cube and $M$ varies on a sphere concentric to the
sphere circumscribed around one of the given polytopes.

Such questions arise frequently as physical problems. The case $\lambda=-1$ is equivalent to finding the points
on the sphere with maximal and minimal potential. There is an extensive list of articles on the topic of finding extremal point configurations on the sphere with respect to some potential function \cite{Energy1}, \cite {Energy2}, and an experimental approach was presented in \cite{Energy3}. Later the question arised to find extremums on the sphere with respect to some potential function and fixed base points.

This proves to be a difficult task, and no complete classification exists. Even in the case when the potential function is the powered Euclidean distance, which we investigate in this article, extremal points vary according to the paramether $\lambda$.

The results given in the currennt article for the case $\lambda=-1$ can also be used to obtain some restrictions about the zeros of the gradient of the electric field, generated by point-charges at the points $A_i$.

We also investigate the values of $\lambda$ for which \eqref{OsSu} is independent of the position of the point $M$ on the sphere. This values correspond to codes with certain strenght.
Early works on this problem characterize the values of $\lambda$ for which the
corresponding sum is independent of the movement of $M$ for the vertices of the regular polygon. Later papers consider the more general problem for multidimensional polytopes.

Of interest is also the situation in which the parameter $\lambda$ is fixed, but the base points are variable. This gives rise to a family of so-called polarization inequalities \cite{Saaf}, \cite{Chebyshev}.
In \cite{Chebyshev}, the
problem is considered for arbitrary $n$ points on the unit circle,
$M$ also on the circle and $\lambda=-2$ and this case is fully
resolved. In \cite {Saaf} the the question is considered in the plane, where the potential is an abitrary convex function on the geodesic distances between points on the circle.

The planar case of the problem we investigate here is consudered in \cite{Pacific}, where a
characterization of the extremal values of the sum \eqref{OsSu} is
given for arbitrary three points and $\lambda\in(0;2)$. The author of that publication also obtains results about the function when the base points are
vertices of the regular $n$-gon and $\lambda\in(0;2n)$. Later this
results are improved and full characterization has been given in
the case of three base points or when the base points are vertices
of a regular polygon in \cite{RN}.

Previous results also include \cite{Pacific2}, where partial results are
obtained when the base points are the vertices of the regular simplex, cross-polytope
and cube when $\lambda\in(0;2)$, using an integral transform from metric geometry. In this paper we give full
characterization and expand those result to all values of
$\lambda$. We give full characterization of the sum \eqref{OsSu}
when $A_i$ are vertices of a regular simplex, a cross-polytope or
a cube and $M$ is a point on a sphere concentric to the sphere
circumscribed around the given polytopes.

In the case of the
regular dodecahedron and icosahedron in $\mathbb{R}^3$ we obtain
results for which values of $\lambda$ the corresponding sum is
independent of the position of $M$ on $\Gamma$.

We begin with consideration of the planar case, as we will later use results from this section.

\section{Planar case}
We begin with a few results for a planar equivalent to the given
problem, as this will be the basis to continue in higher
dimensions.

Let $A_i, i=1,\ldots,n,$ be the vertices of a regular $n$-gon
inscribed in the unit circle. Now assume that $\Gamma$ is a circle
concentric to the circumscribed circle. Put
$B_i=OA_i\bigcap\Gamma$, where $O$ is the center of the $n$-gon.

Let $X\in\Gamma$ be a variable point and
$$R_n(X,\lambda)=\sum_{i=1}^n(\sqrt{|XA_i|^2+h})^{\lambda},$$
where $h\geq0$ i some fixed real number. In a previous article
\cite{RN} the authors have given full characterization of
$R_n(X,\lambda)$ in terms of $\lambda$ when $h=0$. It is easy and
straightforward to modify the proof given there to verify that the
following Theorem holds for all $h\geq0$.

\begin{theorem}\label{OT}
\begin{enumerate}

\item $\lambda<0$. The minimum of $R_n(X,\lambda)$ is achieved
when $X$ bisects the arc between consecutive vertices of
$B_1\ldots B_n$ and the maximum when $X\equiv B_i$. This function
is not bounded in the case when $\Gamma$ is the circumscribed
circle around $A_1\ldots A_n$ and $h=0$.($X\to B_i$ for some $i$).

\item $0\leq\lambda< 2n$. If  $\lambda$ is an even integer, then
$R_n(X,\lambda)$ is independent of the position of $X$ on
$\Gamma$.\\ Otherwise let $m$ be such an integer, that
$2m\leq\lambda\leq2m+2$.\\ If $m$ is even (odd) then
$R_n(X,\lambda)$ is maximal (minimal) if and only if $X$ bisects
the arc between consecutive vertices of $B_1\ldots B_n$. Moreover
$R_n(X,\lambda)$ is minimal (maximal) if and only if $M\equiv
B_i$.

\item $2n\leq\lambda$. The maximum (minimum) of $R_n(X,\lambda)$
is obtained when $X$ coincides with one of the vertices of the
$B_1\ldots B_n$ when $n$ is even (odd) and the minimum (maximum)
is achieved when $X$ bisects the arc between consecutive vertices
when $n$ is even (odd).

\end{enumerate}
\end{theorem}

\begin{remark}

\emph{The term $h$ can be interpreted in the following way. Assume
that $\Gamma$ and the points $A_i$ belong to different
two-dimensional planes at distance $\sqrt{h}$. Let the extremal value of
$\sum_{i=1}^{n}|MA_i|^{\lambda}$ be attained at some point $M_0$.
Then $\sum_{i=1}^{n}|M_0'A_i|^{\lambda}$ is an extremal value of
$\sum_{i=1}^{n}|M'A_i|^{\lambda}$, where $M'$ is the projection of
$M$ in the plane containing $A_i$.}
\end{remark}

\begin{remark}
\emph{It can be proved that
$\sum_{i=1}^n|PA_i|^{2k}=\sum_{j=0}(R+r)^{k-2j}R^kr^k\binom{k}{2j}\binom{2j}{j}$
where $R$ and $r$ are the radii of $\Gamma$ and the circumscribed
circle around the polygon when $k\in\{2,4,\ldots,2n-2\}$. When
$R=r$ this sum transforms to $\binom{2k}{k}nr^{2m}$. The proof is
straightforward from the use of complex numbers and is irrelevant
to the current work, so we shall not present it here. }
\end{remark}

It has been proved that $\{2,4,\ldots,2(n-1)\}$ are the only
integers $j$ for which the sum $\sum_{i=1}^n |PA_i|^j$ is
independent of the position of $P$ on $\Gamma$. By Theorem
\ref{OT} it follows that these are the only real values of
$\lambda$ with this property.

As it turns out this is a characteristic property of the regular polygon.

\begin{theorem}
Given $n$ different points $A_1,A_2,\ldots,A_n$ in the plane  and
a circle $\Gamma$ such that $\sum_{i=1}^n |PA_i|^{2k}$ is
independent of the position of $P$ on $\Gamma$ for
$k\in\{1,2,\ldots,n-1\}$. Then these points are the vertices of a
regular polygon, inscribed in a circle concentric to $\Gamma$.
\end{theorem}

\begin{proof}
We consider the problem in the complex plane and we assume
$\Gamma$ to be the unit circle. We assign complex numbers $a_1,\ldots a_n$ to $A_1,\ldots, A_n$ respectively.
Using the complex polynomial method we obtain
$$\sum_{i=1}^n|x-a_i|^{2k}=\sum_{i=1}^n(x-a_i)^k(\overline{x-a_i})^{k}=\sum_{i=1}^n(x-a_i)^k(\frac{1}{x}-\overline{a_i})^k=c.$$
After multiplying out we obtain
$$(x-a_i)^k(\frac{1}{x}-\overline{a_i})^k=\sum_{j=-k}^k c_{ij} x^j= P_i(x).$$
Now we have
$$\sum_{i=1}^n|x-a_i|^{2k}=\sum_{i=1}^nP_i(x)=\sum_{j=-k}^k \big(\sum_{i=1}^nc_{ij}\big)x^j=c$$
and after multiplying by $x^k$ we get:
$$\sum_{j=-k}^k \big(\sum_{i=1}^nc_{ij}\big)x^{j+k}-cx^k=0.$$
This polynomial has infinitely many zeros (all $x$ with $|x|=1$)
and so it is identically zero. In particular, we have that
$\big(\sum_{i=1}^nc_{ik}\big)=\sum_{i=1}^n\overline{a_i}^k=0$. It
follows that $\sum_{i=1}^na_i^k=0$ for all $k\in\{1,\ldots,n-1\}$.
Using Newton's identities this implies the desired result.
\end{proof}

We may pose the following

\begin{conjecture}\label{open1}
Given $n$ different points $A_1,A_2,\ldots,A_n$ in the plane  and
a circle $\Gamma$ such that $\sum_{i=1}^n |PA_i|^{2n-2}$ is a
constant function of $P\in\Gamma$. Then these points are the
vertices of a regular polygon, inscribed in a circle concentric to
$\Gamma$.
\end{conjecture}

For $n=2$ this conjecture is trivial. To conform it for $n=3,$ we
shall need the following proposition (which is also of independent
interest).

\begin{proposition}\label{n2}
Let $A_1,\ldots, A_n$ be points in the plane, which belong to a circle $T$. Assume that $\Gamma$
is a circle, concentric to $T$, such that $\sum_{i=1}^n|XA_i|^{2k}$, where $k>\Big[\frac{n}{2}\Big]$,
is independent of the position of $X$ on $\Gamma$, then $A_1,\ldots, A_n$ are the vertices of a regular $n-gon$.
\end{proposition}

\begin{proof}
We assign complex numbers $a_i$ to the points $A_i$ and $x$ to
$X$. After a dilatation and a rotation we may assume that $T$ is
the unit circle and that $\prod_{i=1}^na_i=1$. We now have
$$\sum_{i=1}^n|x-a_i|^{2k}=\sum_{i=1}^n(x-a_i)^k(\overline{x}-\overline{a_i})^k=const.$$
Using the same approach as before and $|a_i|=|a_j|,
i,j\in\{1,\ldots,n\}$ we obtain that $\sum_{i=1}^na_i^{t}=0$ for
$t=1,\ldots, k$. If $k\geq n-1$ we easily obtain that $A_1,\ldots,
A_n$ are the vertices of a regular $n$-gon, but from here it
follows that $\sum_{i=1}^na_i^{t}\neq0$, when $t$ is a multiple of
$n$ and hence $k=n-1$. Now if $\Big[\frac{n}{2}\Big]\leq k<n-1$ we
have by Newton's identities that $e_i(a_1,\ldots, a_n)=0$ for
$i=1,\ldots, k$. Now from $|a_i|=1$ and $\prod_{i=1}^na_i=1$ we
have that if $A\subset\{1,\ldots n\}$, then
$$\overline{\prod_{i\in A}a_i}=\prod_{i\in A}\overline{a_i}=\prod_{i\in A}\frac{1}{a_i}=
\frac{1}{\prod_{i\in A}a_i}=\prod_{i\notin A}a_i.$$ Thus we have
$\overline{e_i(a_1,\ldots, a_n)}=e_{n-i}$ and hence
$e_i(a_1,\ldots, a_n)=0$ for $i=1,\ldots, n-1$ as
$\Big[\frac{n}{2}\Big]\leq k$. The conclusion easily follows.
\end{proof}

\begin{proposition} Let $A_1$, $A_2$ and $A_3$ be three different points in the plane.
Assume that there exists a circle $\Gamma$ such that
$|XA_1|^4+|XA_2|^4+|XA_3|^4$ is independent of the position of $X$
on $\Gamma$. Then $A_1$, $A_2$ and $A_3$ are the vertices of an
equilateral triangle.
\end{proposition}
\begin{proof}
We shall use complex numbers. After a dilatation we can consider
$\Gamma$ to be the unit circle. We assign complex numbers $a_i$ to
the points $A_i$ and $x$ to $X$. Assume that some of $a_i$ is
zero, say $a_1=0$. This is equivalent to $|XA_2|^4+|XA_3|^4$ is
independent of the position of $X$ on $\Gamma$, but this is not an
equation of a circle, unless $A_1\equiv A_2\equiv A_3$, which is
not the case. Now we have that
$$\sum_{i=1}^3|XA_i|^4=\sum_{i=1}^3(x-a_1)^2(\overline{x}-\overline{a_1})^2=c.$$

Then
$x^2\sum_{i=1}^3(x-a_1)^2(\overline{x}-\overline{a_1})^2-x^2c=0$.
As this polynomial in $x$ is zero for all $|x|=1$ we have that it
is identically zero. Hence $a_1^2+a_2^2+a_3^2=0$ and
$(1+|a_1|^2)a_1+(1+|a_2|^2)a_2+(1+|a_3|^2)a_3=0$. Put
$1+|a_1|^2=p, 1+|a_2|^2=q, 1+|a_3|^2=r$. Thus we have that
$pa_1+qa_2+ra_3=0$ and then
$$a_1^2+a_2^2+\Big(\frac{pa_1+qa_2}{r}\Big)^2=0.$$ This is
equivalent to $a_1^2(r^2+p^2)+a_2^2(r^2+q^2)+2pqa_1b_2=0$. After
dividing by $a_2^2\neq0$ we obtain a quadratic equation in
$\frac{a_1}{a_2}$ with solutions
$$\frac{a_1}{a_2}=\frac{-pq\pm ir\sqrt{p^2+q^2+r^2}}{r^2+p^2}.$$
Analogously we obtain
$$\frac{a_3}{a_2}=\frac{-qr\pm ip\sqrt{p^2+q^2+r^2}}{r^2+p^2}.$$
Now $$\Big|\frac{a_1}{a_2}\Big|=\frac{p^2q^2+r^2(p^2+q^2+r^2)}{(r^2+p^2)^2}=\frac{q^2+r^2}{p^2+r^2}$$
and
$$\Big|\frac{a_3}{a_2}\Big|=\frac{r^2q^2+p^2(p^2+q^2+r^2)}{(r^2+p^2)^2}=\frac{q^2+p^2}{p^2+r^2}.$$
We have
$$\frac{p-1}{q-1}=\frac{q^2+r^2}{p^2+r^2},\quad \frac{r-1}{q-1}=\frac{q^2+p^2}{p^2+r^2},$$
It follows that
$$p-r=(q-1)\frac{r^2-p^2}{p^2+r^2}\Rightarrow(p-r)\Big(1+(q-1)\frac{p+r}{p^2+r^2}\Big)=0,$$
and thus $p=r$ as $p,q,r>1$.

In the same manner we obtain $p=q=r$ and thus $|a_1|=|a_2|=|a_3|$. Now the result follows from Proposition \ref{n2}.
\end{proof}

We may also confirm Conjecture \ref{open1} if the circle varies.

\begin{proposition}
Given $n$ different points $A_1,A_2,\ldots,A_n$ in the plane  and
a circle $\Gamma$ such that $\sum_{i=1}^n |PA_i|^{2n-2}$ is a
constant function of $P$ when $P$ belongs to an arbitrary
circumference $\Gamma_r$ with radius $r$, concentric to $\Gamma$.
Then these points are the vertices of a regular polygon, inscribed
in a circle concentric to $\Gamma$.
\end{proposition}

\begin{proof}
We consider the problem in the complex plane and we assume
$\Gamma$ to be the unit circle. We assign complex numbers
$a_1,\ldots, a_n$ to $A_1,\ldots, A_n$ respectively. Again using
the complex polynomial method we obtain
$$\sum_{i=1}^n|x-a_i|^{2n-2}=\sum_{i=1}^n(x-a_i)^{n-1}(\overline{x-a_i})^{n-1}=\sum_{i=1}^n(x-a_i)^{n-1}(\frac{r}{x}-\overline{a_i})^{n-1}=c.$$
After multiplying out we obtain
$$(x-a_i)^{n-1}(\frac{r}{x}-\overline{a_i})^{n-1}=\sum_{j=-n+1}^{n-1} P_{ij}(r) x^j,$$
where $P_{ij}(r)$ is some polynomial in $r$. Now we have
$$\sum_{i=1}^n|x-a_i|^{2n-2}=\sum_{j=-n+1}^{n-1} \sum_{i=1}^{n-1}P_{ij}(r) x^j=c$$
and after multiplying by $x^{n-1}$ we get:
\begin{equation}\label{comnumbstep}\sum_{j=-n+1}^{n-1} \sum_{i=1}^{n-1}P_{ij}(r) x^{j+n-1}-cx^{n-1}=0\end{equation}
We fix $r$ and consider this as a polynomial in $x$. It has
infinitely many zeros (all $x$ with $|x|=1$) and so it is
identically zero. In particularly we have that
$\sum_{i=1}^{n-1}P_{ij}(r)=0$. This holds for all $r>0$, so this
polynomial in $r$ has to be identically zero. It is easy to see
that for $j>0$ the leading term of $P_{ij}$ equals
$r^{n-1-j}\overline{a_i}^j$ and hence the leading term in
$\sum_{i=1}^{n-1}P_{ij}(r)$ equals
$\sum_{i=1}^{n-1}\overline{a_i}^jr^{n-1-j}$. From here it follows
that $\sum_{i=1}^{n-1}\overline{a_i}^j=0$ for $j=1,2\ldots, n-2$.
It also remains to note that the leading coefficient of
\eqref{comnumbstep} as a polynomial in $x$ is
$\sum_{i=1}^{n-1}\overline{a_i}^{n-1}x^{n-1}$, hence this sum is
zero and as before the proof is complete.
\end{proof}

\section{Lemmas}
Before we continue with higher dimensional analogs of the
considered problems we need two auxiliary results.

The following lemma is stated and proved in \cite{RN}, but nevertheless we include it here to keep the article self-contained.

\begin{lemma} \label{OL} Let $a_1, a_2,\ldots, a_n$  and $b_1, b_2,\ldots, b_n$ be real numbers and $b_i, i=1,\ldots, n$
be nonnegative, then the function
$$\Theta(\lambda)=\sum_{i=1}^na_ib_i^{\lambda}$$ is either
identically zero or has at most $n-1$ real solutions for $\lambda$
counted with their multiplicities.
\end{lemma}
\begin{proof}
We proceed by induction on the number of summands. For $n=1$ we
have that $ab^{\lambda}=0$, which does not have solutions if both
of $a$ and $b$ are nonzero. If either of them is zero then
$ab^{\lambda}$ is identically zero. Now assume the statement to be
true for all $k<n$.

For $k=n$ if either of $a_i$ or $b_i$ is zero then we use the
induction hypothesis. Now let $b_i, a_i$ be nonzero. As all of
$b_i$ are nonzero then we can divide each term by $b_1^{\lambda}$
to get $$\sum_{i=1}^na_i\Big(\frac{b_i}{b_1}\Big)^{\lambda}=0.$$
Assume that this equation is not identically zero and its
solutions are $y_1,\ldots,y_k$ with multiplicities $t_1,\ldots,
t_k$ and $\sum_{i=1}^k t_i>n-1$. Differentiating this with respect
to $\lambda$ we get
$$\sum_{i=2}^na_i\ln\Big(\frac{b_i}{b_1}\Big)\Big(\frac{b_i}{b_1}\Big)^{\lambda}=0=\sum_{i=2}^n a_i^{'} b_i^{'\lambda},$$
where $a_i'=a_i\ln\Big(\frac{b_i}{b_1}\Big)$ and
$b_i'=\frac{b_i}{b_1}$. Assume that this expression is identically
zero, then $\sum_{i=1}^na_ib_i^{\lambda}=0$ must be a constant,
and the claim follows. Assume that the derivative does not vanish
for all $\lambda$. Now by the induction hypothesis the derivative
has at most $n-2$ zeros. But we have that $y_1,\ldots,y_k$ are
solutions to the above equation with multiplicities $t_1-1,\ldots,
t_k-1.$ Moreover, by Rolle's theorem the derivative has at least
one root in each interval $(y_i;y_{i+1})$, and thus we obtain
$k-1+\sum_{i=1}^k t_i-1$ solutions (counted with their
multiplicities), which is greater than $n-2$- a contradiction. It
follows that $\sum_{i=1}^k t_i\leq n-1$. The lemma is proved.
\end{proof}

\begin{lemma}\label{hcon}Let $\pi$ and $\xi$ be two nonparallel hyperplanes
in $\mathbb{R}^n$, $n>2$. For every two points $\mathbf{x}$ and
$\mathbf{y}$  belonging to the sphere $S^{n-1}$ there exists a
sequence $S_1,\ldots, S_k$ of two-codimensional spheres such that:

\begin{enumerate}
\item $\mathbf{x}\in S_1$ and $\mathbf{y}\in S_k$,
\item $S_i\bigcap S_{i+1}\neq\emptyset$,
\item $S_i=S^{n-1}\bigcap\alpha$, where $\alpha$ is a hyperplane parallel either to $\pi$ or to $\xi$.
\end{enumerate}
\end{lemma}

We shall prove more.

\begin{lemma} Let $M$ be  a $C^1$-smooth compact connected
surface in $\mathbb{R}^n$, $n>2$ and let $\alpha\nparallel\beta$ be
hyperplanes. Then for any points $A,B\in M$ one may find points
$C_0=A,C_1,\dots,C_{n-1},C_n=B$ on $M$ such that
$C_{i-1}C_i\parallel\alpha$ or $C_{i-1}C_i\parallel\beta,$ $1\le
i\le n.$
\end{lemma}
\begin{proof}After a non-singular linear transformation, may
assume that $\alpha\perp O{x_1}$ and $\beta\perp O{x_2}.$ By the
Jordan-Brouwer separation theorem, $M$ is the boundary of a
bounded domain $\Omega$ in $\mathbb {R}^n.$ We may find point
$D_0=A,D_1,\dots,D_{n-1},D_n=B\in\Omega$ such that $D_{i-1}D_i$ is
parallel to a coordinate axis, $1\le i\le n.$ Then for any $1\le
i\le n$ there exist a hyperplane $\gamma_i\supset D_{i-1}D_i$
parallel to $\alpha$ or $\beta.$ Since $D_i\in l_i\cap\Omega,$
where $l_i=\gamma_i\cap\gamma_{i+1},$ one may find point $C_i\in
l_i\cap M,$ $1\le i\le n-1.$ These points have the desired
property. The reason this argument fails in dimensions 1 and 2 is that the intersection of two 1 co-dimensional planes in $\mathbb{R}$ or $\mathbb{R}^2$ is discrete.
\end{proof}

We are now ready to continue with the Platonic solids.

\section{Dodecahedron}
We shall use the numbering of vertices presented in the plane projection of the dodecahedron below.

\begin{figure}[h]
\includegraphics[scale=0.7]{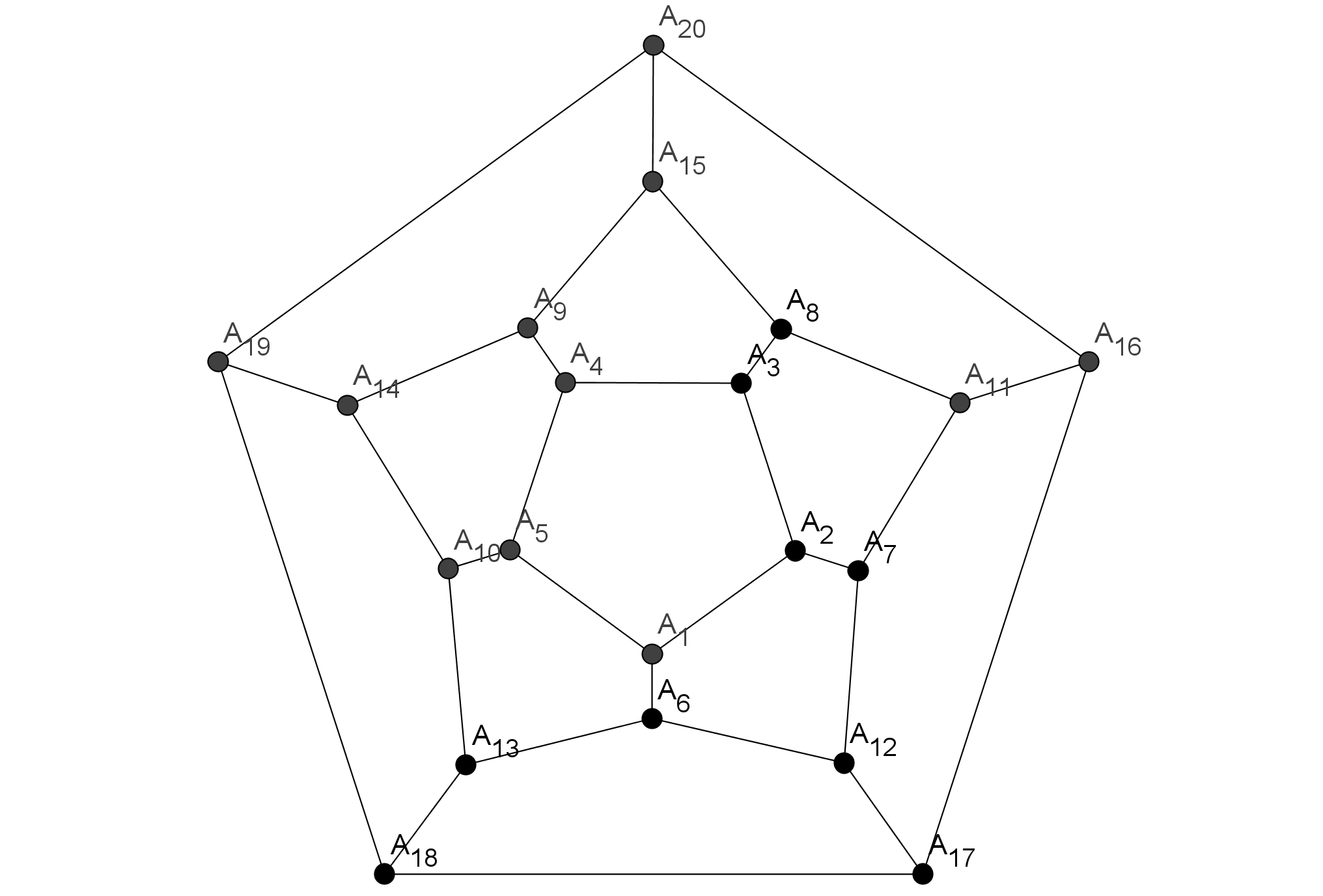}
\caption{Plane projection of a regular dodecahedron}
\label{dodecapic}
\end{figure}

We investigate the following question: Given a regular
dodecahedron $A_1,\ldots A_{20}$ and $\Gamma$ the circumscribed
sphere of the polytope, determine all real numbers $\lambda$ such
that
\begin{equation}\label{dodeca}\sum_{i=1}^{20}MA_i^{\lambda}\end{equation}
is independent of the position of $M$ on $\Gamma$.

First we shall prove the following

\begin{proposition}\label{dprop1} There are at most eight numbers $\lambda$ for
which the sum $\sum_{i=1}^{20}MA_i^{\lambda}$ is independent of
the position of $M$ on $\Gamma$.
\end{proposition}

\begin{proof}
Let $A_1A_2A_3A_4A_5$ be a face of the dodecahedron and let $A$ be
the center of the face. Define $C=OA\bigcap\Gamma,$ where $O$ is
the center of $\Gamma$. Then for all numbers $\lambda$ for which
\eqref{dodeca} is constant we have

\begin{equation}\label{eqstepeni}\sum_{i=1}^{20}CA_i^{\lambda}=\sum_{i=1}^{20}A_1A_i^{\lambda}.\end{equation}
But we have that $|CA_i|=:d_1$ for $i\in\{1,\ldots,5\}$,
$|CA_i|=:d_2$ for $i\in\{6,\ldots,10\}$, and so on $|CA_i|=:d_4$
for $i\in\{16,\ldots,20\}$. We also have $|A_1A_i|=:d_5$ for
$i\in\{6,5,2\}$, $|A_1A_i|=:d_6$ for $i\in\{3,4,7,10,13,12\}$,
$|A_1A_i|=:d_7$ for $i\in\{8,9,11,14, 17, 18\}$ and
$|A_1A_i|=:d_8$ for $i\in\{15, 16, 19\}$, $|A_1A_{20}|=:d_9$
(Figure \ref{dodecapic}). Then from the fact that \ref{dodeca} is
a constant function of $M$ on $\Gamma$ we have
$$\sum_{i=1}^45d_i^{\lambda}=3d_5^{\lambda}+6d_6^{\lambda}+6d_7^{\lambda}+3d_8^{\lambda}+d_9^{\lambda},$$ or
$$5d_1^{\lambda}+5d_2^{\lambda}+5d_3^{\lambda}+5d_4^{\lambda}-3d_5^{\lambda}-6d_6^{\lambda}-6d_7^{\lambda}-3d_8^{\lambda}-d_9^{\lambda}=0.$$
Now it is easy to see that this cannot hold for all $\lambda$ as $d_9>d_i$ for $i\neq9$ and we have that when
$\lim_{\lambda\to\infty} 5d_1^{\lambda}+5d_2^{\lambda}+5d_3^{\lambda}+5d_4^{\lambda}-3d_5^{\lambda}-6d_6^{\lambda}-6d_7^{\lambda}-3d_8^{\lambda}-d_9^{\lambda}=-\infty.$
Hence by Lemma \ref{OL} we have that there are at most eight real number $\lambda$ for which \eqref{dodeca} is constant.
\end{proof}

\begin{proposition}\label{dprop2} The sum \eqref{dodeca} is constant for $\lambda=2,4,6,8,10$.
\end{proposition}
\begin{proof}
Consider the plane $\pi$, which contains the points
$A_1,A_2,A_3,A_4,A_5$, $\pi_2$, which contains the points
$A_6,A_7,A_8,A_9,A_{10}$, $\pi_3$, which contains the points
$A_{11},A_{12},A_{13},A_{14},A_{15}$, $\pi_4$, which contains the
points $A_{16},A_{17},A_{18},A_{19},A_{20}$. First we notice that
the planes $\pi_i$ are parallel. Second in each of the planes
$\pi_i$ the points $A_{5i-4},A_{5i-3}, A_{5i-2}, A_{5i-1}, A_{5i}$
are vertices of a regular pentagon, moreover the centers of these
polygons lie on a line trough the center of the dodecahedron. Now
consider a plane $\alpha$, parallel to $\pi_1$ and take
$\omega=\alpha\bigcap\Gamma$. Now we have that for each
$M\in\omega$
$$\sum_{i=1}^{20}|MA_i|^{\lambda}=\sum_{i=1}^{5}|MA_i^{'2}+h_1^2|^{\frac{\lambda}{2}}+
\sum_{i=6}^{10}|MA_i^{'2}+h_2^2|^{\frac{\lambda}{2}}+\sum_{i=11}^{15}|MA_i^{'2}+h_3^2|
^{\frac{\lambda}{2}}+\sum_{i=16}^{20}|MA_i^{'2}+h_4^2|^{\frac{\lambda}{2}},$$
where $h_i$ is the distance between the planes $\pi_i$ and
$\alpha$ and $A_i^{'}$ is the projection of the point $A_i$ to the
plane $\alpha$. Now by Theorem \ref{OT} for
$\lambda\in\{2,4,6,8\}$ each of the sums
$\sum_{i=1}^{5}|M_1A_i^{'2}+h_1^2|^{\frac{\lambda}{2}},\sum_{i=6}^{10}|M_2A_i^{'2}+h_2^2|
^{\frac{\lambda}{2}},\sum_{i=11}^{15}|M_3A_i^{'2}+h_3^2|^{\frac{\lambda}{2}},\sum_{i=16}^{20}|M_4A_i^{'2}+h_4^2|^{\frac{\lambda}{2}}$
is constant.

Now for $\lambda=10$ we have that
$$\sum_{i=1}^{5}|MA_i^{'2}+h_1^2|^{5}=|MA_i^{'}|^{10}+c_4\sum_{i=1}^5|MA_i^{'}|^8+
c_3\sum_{i=1}^5|MA_i^{'}|^6+c_2\sum_{i=1}^5|MA_i^{'}|^4+c_1\sum_{i=1}^5|MA_i^{'}|^2+c_0,$$
and as we know $c_j\sum_{i=1}^5|MA_i^{'}|^{2j}$ is constant for
$j=1,2,3,4$. Hence we only need to prove that
$\sum_{i=1}^{20}|MA_i^{'}|^{10}$ is constant. But
$A_1A_2A_3A_4A_5A_{16}A_{17}A_{18}A_{19}A_{20}$ and
$A_{6}A_{7}A_8A_9A_{10}A_{11}A_{12}A_{13}A_{14}A_{15}$ are
vertices of two regular decagons and hence by the Theorem \ref{OT}
we have that \eqref{dodeca} is constant for $\lambda=10$ and
$M\in\omega$.

Now by Lemma \ref{hcon} this result is easily extended to the whole sphere.
\end{proof}

We can limit the values of $\lambda$ with the desired property.
\begin{proposition}\label{dprop3}
All the values of $\lambda$ for which \eqref{dodeca} is independent of the position of $M$ on $\Gamma$ are among $2,4\ldots,18$.
\end{proposition}

\begin{proof}
Let $\pi$ be a plane trough the center of $\Gamma$, parallel to the face $A_1A_2A_3A_4A_5$.

Let $\omega=\pi\bigcap\Gamma$ and $A_i'$ be the projection of the
point $A_i$ in $\pi$. For $M\in\Gamma$ we have
$$\sum_{i=1}^{20}|MA_i|^{\lambda}=\sum_{i=1}^5|MA_i'^2+h_1^2|^{\lambda/2}+\sum_{i=16}^{20}
|MA_i'^2+h_1^2|^{\lambda/2}+\sum_{i=6}^{15}|MA_i'^2+h_2^2|^{\lambda/2},$$
where $h_1$ is the distance between $\pi$ and the planes
containing the faces $A_1A_2A_3A_4A_5$ and
$A_{16}A_{17}A_{18}A_{19}A_{20}$ and $h_2$ is the distance between
$\pi$ and the planes containing the vertices $A_6, A_7, A_8, A_9,
A_{10}$ and $A_{11}, A_{12}, A_{13}, A_{14}, A_{15}$. Now
$A_1'A_2'A_3'A_4'A_5'A_{16}'A_{17}'A_{18}'A_{19}'A_{20}'$ and \\
$A_6'A_7'A_8'A_9'A_{10}'A_{11}'A_{12}'A_{13}'A_{14}'A_{15}'$ are
two regular homothetic dodecagons. By Theorem \ref{OT} the sum
\eqref{dodeca} is independent of the position of $M$ on
$\omega=\Gamma\bigcap\pi$ for $\lambda=2,4\ldots,18$ and these are
the only powers with this property(using Theorem \ref{OT} for each
of the decagons independently).
\end{proof}

\begin{remark}
\emph{If $\Gamma$ is a sphere concentric to the sphere
circumscribed around the regular dodecahedron we can still
consider these questions. If we apply the same approach as in the
proof of Proposition \ref{dprop1} we can obtain that there are at
most nine real $\lambda$ with the desired property. Again we
consider the equation \eqref {eqstepeni}, but instead of point
$A_1$ we consider the point $A_1'=OA_1\bigcap\Gamma$. Again using
Lemma \ref{OL} the result follows. Propositions \ref{dprop2} and
\ref{dprop3} still hold in this case, the proofs being analogous.}
\end{remark}

\section{Icosahedron}
Now we begin with the consideration of the icosahedron. We shall use the
numbering of vertices presented in the plane projection of the icosahedron below.

\begin{figure}[h]

\includegraphics[scale=0.8]{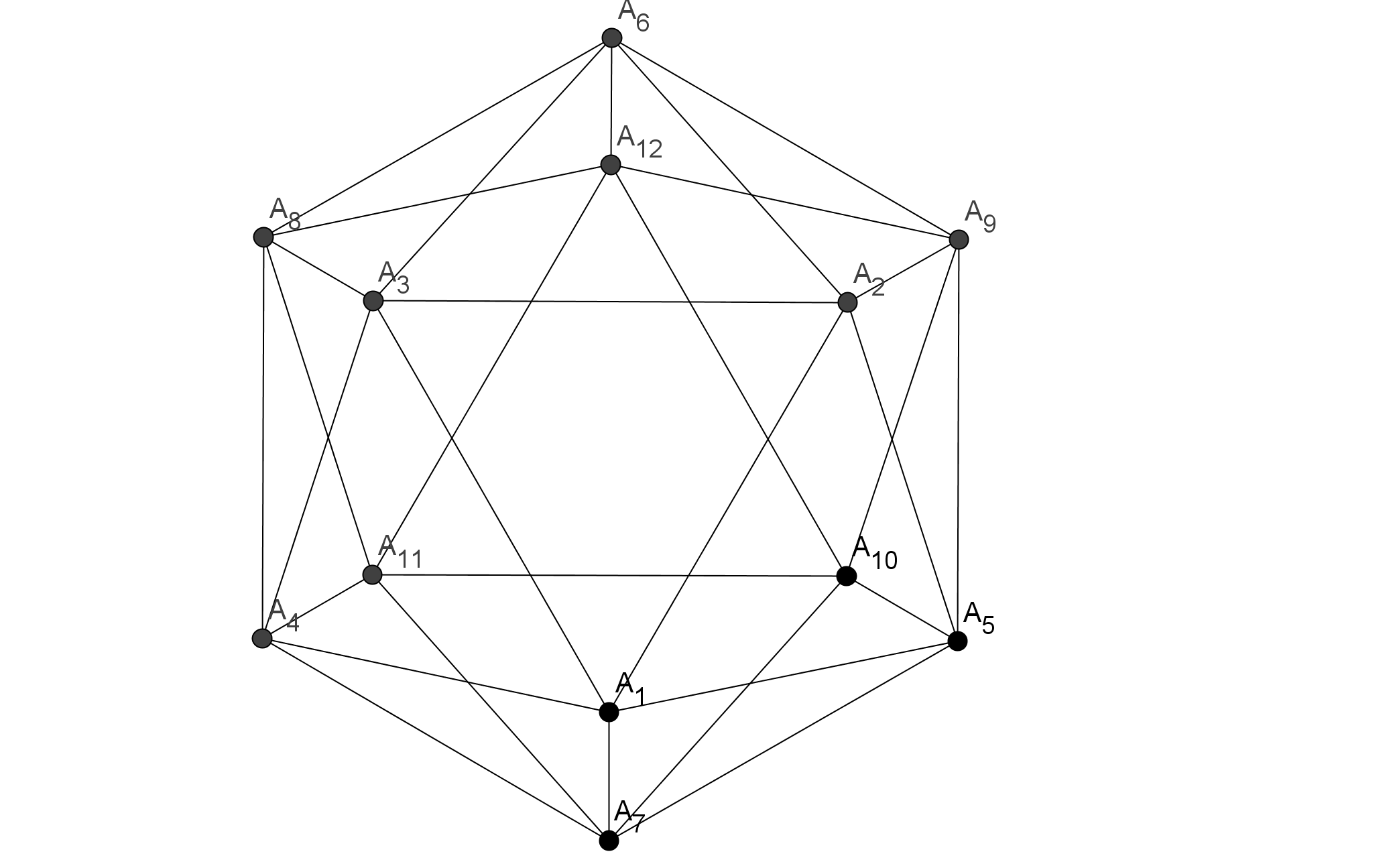}
\caption{Plane projection of a regular icosahedron}

\label{icosapicture}
\end{figure}

Assume that $A_i, i\in\{1,\ldots,12\}$ are the vertices of a regular icosahedron.
Let $\Gamma$ be a sphere concentric to the circumscribed sphere of the polytope.
We investigate the question for which $\lambda$ \begin{equation}\label{icosa}
\sum_{i=1}^{12}|A_iM|^{\lambda}\end{equation}
is constant for every $M\in\Gamma$.

\begin{proposition}\label{iprop2}
The sum (\ref{icosa}) is constant for $\lambda=2,4,6$.
\end{proposition}

\begin{proof} Let $\pi$ be a plane parallel to the plane, containing $A_1A_2A_3$ and
such that $\pi\bigcup\Gamma=\omega$ is a circle. Let $h_1, h_2,
h_3, h_4$ be the distances from the planes containing $A_1A_2A_3$,
$A_4A_5A_6$, $A_7A_8A_9$ and $A_{10}A_{11}A_{12}$ to $\pi$.
Obviously these planes are all parallel to $\pi$.

Now we have that
\begin{equation}\label{prop2icosa}\sum_{i=1}^{12}|A_iM|^{\lambda}=\sum_{i=1}^{12}
(|A_i'M|^2+h_{\left\lceil\frac{i}{3}
\right\rceil}^2)^{\frac{\lambda}{2}},\end{equation} where
$M\in\omega$ and  $A_i'$ is the projection of $A_i$ in $\pi$ since
we have $|A_iM|^2=|A_i^{'}M|^2+h_{\left\lceil\frac{i}{3}
\right\rceil}^2$.

Now for $\lambda=2$ we have the sum
$$\sum_{i=1}^{12}(|A_i^{'}M|^2+h_{\left\lceil\frac{i}{3}
\right\rceil}^2)^{\frac{\lambda}{2}}=
\sum_{i=1}^4|MA_{3i}^{'}|^2+|MA_{3i-1}^{'}|^2+|MA_{3i-2}^{'}|^2+3h_i^2,$$
which is constant since from the planar case we know that
$|MA_{3i}^{'}|^2+|MA_{3i-1}^{'}|^2+|MA_{3i-2}^{'}|^2$ is constant.

For $\lambda=4$ we have the sum
$$\sum_{i=1}^{12}(|A_i^{'}M|^2+h_{\left\lceil\frac{i}{3}
\right\rceil}^2)^{\frac{\lambda}{2}}=
\sum_{i=1}^4|MA_{3i}^{'}|^4+|MA_{3i-1}^{'}|^4+|MA_{3i-2}^{'}|^4+$$$$+2h_i^2
(|MA_{3i}^{'}|^2+|MA_{3i-1}^{'}|^2+|MA_{3i-2}^{'}|^2)+3h_i^4,$$
which is constant since from the planar case we know that each of
$|MA_{3i}^{'}|^4+|MA_{3i-1}^{'}|^4+|MA_{3i-2}^{'}|^4$ and
$|MA_{3i}^{'}|^2+|MA_{3i-1}^{'}|^2+|MA_{3i-2}^{'}|^2$ are
constant($A_1A_2A_3$, $A_4A_5A_6$, $A_7A_8A_9$ and
$A_{10}A_{11}A_{12}$ all project to equilateral triangles).

For $\lambda=6$ we have
$$\sum_{i=1}^{12}(|A_i^{'}M|^2+h_{\left\lceil\frac{i}{3}
\right\rceil}^2)^{\frac{\lambda}{2}}=
\sum_{i=1}^{12}|A_i^{'}M|^6+3h_{\left\lceil\frac{i}{3}
\right\rceil}^2|A_i^{'}M|^4+3h_{\left\lceil\frac{i}{3}
\right\rceil}^4|A_i^{'}M|^2+h_{\left\lceil\frac{i}{3}
\right\rceil}^6.$$

As we know each of $\sum_{i=1}^{12}|A_i^{'}M|^{\lambda}$ is
constant for $\lambda=2,4$ we need only prove that
$\sum_{i=1}^{12}|A_i^{'}M|^{6}$ is independent on the position of
$M$ on $\omega$. This follows directly from the consideration of
the planar case and the fact that
$A_1^{'}A_2^{'}A_3^{'}A_{10}^{'}A_{11}^{'}A_{12}^{'}$ and
$A_4^{'}A_5^{'}A_6^{'}A_7^{'}A_8^{'}A_9^{'}$ are vertices of
regular hexagons. Thus we have obtained that \eqref{icosa} is
independent of the position of $M$ on $\omega$ where $\omega$ is a
circle, obtained by intersecting a plane, parallel
to a plane containing a face of the icosahedron and $\Gamma$. \\

Now again from Lemma \ref{hcon} the proposition easily follows.
\end{proof}

\begin{proposition}\label{iprop3}
All the powers for which \eqref{icosa} is constant are among $2,4\ldots, 10$.
\end{proposition}

\begin{proof}
Let $\pi$ be a plane trough the center of $\Gamma$, parallel to
the face $A_1A_2A_3$. Let $\omega=\pi\bigcap\Gamma$ and $A_i'$ be
the projection of the point $A_i$ in $\pi$. For $M\in\Gamma$ we
have
$$\sum_{i=1}^{12}|MA_i|^{\lambda}=\sum_{i=1}^3(MA_i'^2+h_1^2)^{\lambda/2}+\sum_{i=4}^{9}
(MA_i'^2+h_1^2)^{\lambda/2}+\sum_{i=10}^{12}(MA_i'^2+h_2^2)^{\lambda/2},$$
where $h_1$ is the distance between $\pi$ and the planes
containing the faces $A_1A_2A_3$ and $A_{10}A_{11}A_{12}$ and
$h_2$ is the distance between $\pi$ and the planes containing the
vertices $A_4, A_5, A_6$ and $A_7, A_8, A_9$. Now
$A_1'A_2'A_3'A_{10}'A_{11}'A_{12}'$ and $A_4'A_5'A_6'A_7'A_8'A_9'$
are two regular homothetic hexagons. By Theorem \ref{OT} the sum
\eqref{icosa} is independent of the position of $M$ on
$\omega=\Gamma\bigcap\pi$ for $\lambda=2,4\ldots,10$ and these are
the only powers with this property(using Theorem \ref{OT} for each
of the hexagons independently).
\end{proof}

\begin{remark}\emph{A computer check suggest that all the powers $\lambda$
for which \eqref{dodeca} is independent of the position of $M$ on
$\Gamma$ are $\lambda=2,4\ldots, 10$ and all the powers $\lambda$
for which \eqref{icosa} is independent of the position of $M$ on
$\Gamma$ are $\lambda=2,4\ldots, 10$. This may be to the fact that
the regular dodecahedron and the regular icosahedron are dual. No
complete mathematical proof of these results is known to the
authors.}
\end{remark}

\section{Higher dimensional regular polytopes}
Now we begin with the consideration of the higher dimensional
regular polytopes. It is a known fact that in the Euclidean spaces
of dimension $n>4$ there exist only three $n$-dimensional regular
polytopes -- the regular simplex, the cross-polytope (orthoplex)
and the hypercube.

We are interested in the following question: For which points are
the extremal values of \begin{equation}\label{sum}\sum_{i=1}^t
|A_iM|^{\lambda}\end{equation} obtained, when $M$ varies on a
sphere, concentric to the sphere on which the points $A_i$ belong
when these points are the vertices of
\begin{itemize}
\item a regular simplex; \item a cross-polytope; \item a
hypercube.
\end{itemize}

\subsection{Regular simplex}
We begin with the consideration of the regular simplex. We shall prove the following:

\begin{theorem}\label{simplex}
Let $A_1\ldots A_{n+1}$ be a regular simplex in $\Bbb R^n$ and let
$\Gamma$ be a sphere concentric to the sphere circumscribed around
the given simplex. Put $B_i=OA_i\bigcap\Gamma$, where $O$ is the
center of $\Gamma$ and
$S_n(X,\lambda)=\sum_{i=1}^{n+1}(|XA_i|^2+h)^{\lambda/2}$, where
$h>0$ is some fixed real number. Then for any positive fixed $h$
we have
\begin{enumerate}

\item \label{simplex1}$\lambda<0$ Then the minimum of
$S_n(X,\lambda)$ is obtained when $X={B_iO}\bigcap\Gamma$ for some
vertex $B_i$. The maximum is obtained when $X\equiv B_i$ for some
$B_i$. When $\Gamma$ is the sphere circumscribed around the
regular simplex this sum is obviously unbounded ($X\to B_i$ for
some $B_i$). \item  \label{simplex2}$\lambda\in[0;4]$ If $\lambda$
is an even integer, then $S_n(X,\lambda)$ is independent of the
position of $M$ on $\Gamma$. Otherwise, let $2m<\lambda<2m+2$. If
$m$ is even (odd), then the maximum (minimum) of $S_n(M,\lambda)$
is obtained when $X=B_iO\bigcap\Gamma$ for some $i$, and the
minimum (maximum) is obtained when $X\equiv B_i$ for some $i$.
\item \label{simplex3}$\lambda>4$ The maximum of $S_n(M,\lambda)$
is obtained when $X=B_iO\bigcap\Gamma$ and the minimum when
$X\equiv B_i$.
\end{enumerate}
\end{theorem}

Previous work on this problem has been done by Stolarsky in
\cite{Pacific2}, who obtains partial characterization of the
extremal values for $\lambda\in(0;2)$ in the case when $\Gamma$ is
the sphere circumscribed around the simplex.

The results obtained for the regular simplex are very similar to
those obtained for the equilateral triangle.

\begin{proof}
We shall first prove part \ref{simplex1} of the theorem. We shall
use induction on the dimension of the simplex. For the planar case
we know that this is true. And that $0,2,4$ are the only powers
for which $S_n(X,\lambda)$ is constant.

Assume now that $A_1,\ldots,A_{n+1}$ are vertices of a regular
simplex in $\Bbb R^n$ and $\Gamma$ a sphere, concentric to the
sphere circumscribed around the given simplex. Take now a vertex
$A_i$ and a hyperplane $\pi$, $\dim\pi=n-1$ such that $\pi$ is
perpendicular to $OA_i$ and $\pi$ intersects $\Gamma$ in such a
way that $\dim\pi\bigcap\Gamma=\Gamma^{'}=n-2$. This is possible
as the sphere is a compact differentiable smooth manifold. Now we
have that
$S_n(X,\lambda)=\sum_{j=1}^{n+1}(|MA_j'|^2+h_j|)^{\lambda/2}$ for
$M\in\Gamma'$ where $A_j'$ is the projection of $A_j$ in the plane
$\pi$ and $h_j$ is the distance between $A_j$ and $\pi$. We have
now that $A_i$ projects to the center of $\Gamma'$, hence $MA_i$
is constant for $M\in\Gamma'$. Moreover as the polytope is a
regular simplex then $h_m=h_n$, $m,n\neq i$ and $A_j^{'}$, $j\neq
i$ are the vertices of a regular $(n-1)$-dimensional simplex.
Hence by the induction hypothesis $S_n(X,\lambda)$ is constant for
$M\in\Gamma^{'}$ and $\lambda=0,2,4$. Now from Lemma \ref{hcon} it
follows that $S_n(M,\lambda)$ is independent on the position of
$M$ on $\Gamma$ when $\lambda=0,2,4$. From the induction
hypothesis it also follows that this are the only values for
$\lambda$ with this property.

For the other cases of Theorem \ref{simplex} we again use induction.

Let $\lambda\neq  0,2,4$. Then as $\Gamma$ is a compact set and
$S_n(M,\lambda)$ is a continuous function, then there is a maximum
of $S_n(M,\lambda)$. Assume that this maximum is achieved at a
point $N$. Now consider the hyperplane $\pi_i$, which contains the
$(n-1)$-dimensional simplex obtained by the vertices
$A_1,\ldots,A_{i-1},A_{i+1},\ldots,A_n$. Now let $\pi_i'$ be the
hyperplane, parallel to $\pi_i$, which contains $N$. We consider
the extremal values of $S_n(M,\lambda)|_{M\in\pi_i'}$.

We have that $|MA_i|$ is constant for $M\in\pi_i'$, so we need
only consider
$$\sum_{j=1,j\neq i}^{n+1}(|MA_j|^2+h)^{\lambda/2}=\sum_{j=1,j\neq i}^{n+1}
(|M'A_j|^2+h_j^2+h)^{\lambda/2}=S_{n-1}(M',\lambda),$$
where $M'$ is the projection of $M$ in $\pi_i$ and $h_m=h_n,
m,n\neq i$. We have that $M'\in\Gamma'$, where $\Gamma'$ is the
projection of $\Gamma\bigcap\pi_i'$ in $\pi_i$. Now by the
induction hypothesis we have that if $O'$ is the projection of $O$
in the hyperplane $\pi_i$ then if the maximum of
$S_{n-1}(M',\lambda)$ is obtained at $M_i$ then
\begin{enumerate}
\item $\lambda<0$ $M_i=\overrightarrow{O_iA_j}\bigcap\Gamma'$ for
some $j\neq i$. In this case if $h_j^2+h=0$, then $\Gamma$ is the
sphere circumscribed around the regular simplex and the sum
$S_n(M,\lambda)$ is not bounded when $\lambda<0$.

\item $\lambda\in(0;2)$ $M_i=\overrightarrow{A_jO_i}\bigcap\Gamma'$ for some $j\neq i$.
\item $\lambda\in(2;4)$ $M_i=\overrightarrow{O_iA_j}\bigcap\Gamma'$ for some $j\neq i$.
\item $\lambda>4$ $M_i=\overrightarrow{A_jO_i}\bigcap\Gamma'$ for some $j\neq i$.
\end{enumerate}

Now as the global maximum of $S_n(M,\lambda)$ is obtained at $N$,
hence we have that the maximum of $S_n(M,\lambda)|_{M\in\pi_i'}$
is also obtained at $N$, hence we have that the projection of $N$
in the hyperplane $\pi_i$ coincides with $M_i$. This remains true
for all planes $\pi_i$.

It remains only to prove that the only points for the projections
of which this holds are the aforementioned in the Theorem. We
shall prove the following Lemma
\begin{lemma}
Let $A_1\ldots A_{n+1}$ be a regular simplex in $\mathbb{R}^n$.
Let $\Gamma$ be a sphere circumscribed around the simplex. For
every $i=1,\ldots,n+1$ let $O_i$ be the projection of the center
of $\Gamma$ in the hyperplane $\pi_i$ containing the face of the
simplex, which does not contain the vertex $A_i$. Let $X$ be a
point on $\Gamma$. Let $X_i$ be the projection of $X$ in $\pi_i$.
If for every $i$ $O_iX_i$ is perpendicular to so some two
codimensional face of the simplex, then either $X\equiv B_i$ for
some $i$ or $X\equiv \overrightarrow{B_iO}\bigcap\Gamma$.
\end{lemma}
\begin{proof}
It is easy to verify this in $\mathbb{R}^3$. Assume now that the
dimensional is greater than $3$. We have that $O_1X_1$ is
perpendicular to so some $n-2$-dimensional face of the simplex,
say the face $A_3\ldots A_{n-1}$. Now it easily follows that
$|X_1A_i|=const$, $i=3,\ldots,n+1$. Hence $|XA_i|=const$,
$i=3,\ldots,n+1$. Making the same considerations in $\pi_3$ we
obtain that $X$ is equidistant from some other $n-2$ face of the
simplex, different than $A_2\ldots A_{n+1}$ and hence we have that
$X$ is equidistant either from all the vertices of $A_1\ldots
A_{n+1}$ and thus $X\equiv O$, but $X\in\Gamma$ and then $X$ is
equidistant from $n-1$ vertices of the simplex. The conclusion
follows easily.
\end{proof}

>From this Lemma and the induction hypothesis the proof of Theorem \ref{simplex} follows.

The proof of the minimality part of Theorem \ref{simplex} is analogous.
\end{proof}

In the light of the planar case one may pose the following

\begin{conjecture}
Let $A_1,\ldots,A_{n+1}$ be points in $\mathbb{R}^n$. Assume that
there is a sphere $\Gamma$ such that
$\sum_{i=1}^{n+1}|MA_i|^{\lambda}$ is independent of the position
of $M$ on $\Gamma$ for $\lambda=2,4$. Then $A_1\ldots A_{n+1}$ are
vertices of a regular simplex.
\end{conjecture}

We have already shown that this is true for $n=2$. It can also be
proved for $n=3$.

\begin{proposition}
Let $A_1,A_2,A_3,A_4$ be points in $\mathbb{R}^3$. Assume that
there is a sphere $\Gamma$ such that
$\sum_{i=1}^4|MA_i|^{\lambda}$ is independent of the position of
$M$ on $\Gamma$ for $\lambda=2,4$. Then $A_1A_2A_3A_4$ is a
regular tetrahedron.
\end{proposition}

\begin{proof}
We shall present the proof here. Consider a plane $\pi$, such that
$\pi\bot OA_1$ and $\pi\bigcap\Gamma=\omega$, a circle. Now we
have that for $M\in\omega$ $|MA_1|$ is constant, so we need only
consider the sum
$|MA_2|^{\lambda}+|MA_3|^{\lambda}+|MA_4|^{\lambda}$. Let $A_i'$
be the projection of the point $A_i$ in the plane $\pi$. we have
that
$|MA_2|^2+|MA_3|^2+|MA_4|^2=|MA_2'|^2+h_2^2+|MA_3'|^2+h_3^2+|MA_4'|^2+h_4^2$,
hence $|MA_2'|^2+|MA_3'|^2+|MA_4'|^2$ is also constant for
$M\in\omega$. We also have that
$|MA_2|^4+|MA_3|^4+|MA_4|^4=|MA_2'|^4+|MA_3'|^4+|MA_4'|^4+2|MA_2'|^2h_2^2+
2|MA_3'|^2h_3^2+2|MA_4'|^2h_4^2+h_2^4+h_3^4+h_4^4$.
Now consider $\pi$ to be the complex plane with origin-the center
of $\omega$. Assign to the points $A_i'$ the complex numbers
$\alpha_i$. Using the same approach, as in the proof of the planar
case we obtain that $\sum\alpha_i=0$ and $\sum\alpha_i^2=0$, from
where it follows that $\alpha_i=z\xi^i$, where $\xi$ is the third
root of unity. From here it follows that $A_i'$ are the vertices
of an equilateral triangle, centered at the origin of the complex
plane. Now again by the considerations at the beginning of this
paper it follows that $|MA_2'|^4+|MA_3'|^4+|MA_4'|$ is constant
for $M\in\omega$, hence
$2|MA_2'|^2h_2^2+2|MA_3'|^2h_3^2+2|MA_4'|^2h_4$ is constant for
$M\in\omega$. It follows that $\sum_{i=2}^42z\xi^ih_i^2=0$. It is
now easy to see that $h_2^2=h_3^2=h_4^2$ and hence $h_1=h_2=h_3$.
Assume that two of the points $A_i$ belong to different
halfspaces, divided by the plane $\pi$, say $A_2$ and $A_3$. Then
considering a plane $\pi'$, such that the distance between $\pi$
and $\pi'$ $\epsilon$, is such that $\pi'$ still divides the
points $A_2$ and $A_3$ and $\pi'\bigcap\Gamma=\omega'$, still a
circle. Then it follows that $h_2+\epsilon=h_3-\epsilon$, which is
impossible. Hence all of the points $A_i$, $i=2,3,4$ belong to one
halfspace, and thus to one plane. From here we get that
$A_2A_3A_4$ is an equilateral triangle. With the same
considerations it follows that every three of the vertices of
$A_1A_2A_3A_4$ are vertices of an equilateral triangle, hence it
is a regular tetrahedron.
\end{proof}

\subsection{Cross-polytope}
Now we begin with the consideration of the cross-polytope. The
vertices of this cross-polytope are of the form
$(\pm1,0,\ldots,0)$ and all permutations. It can also be
considered as the unit sphere under the $l_1$ metric. It is the
dual of the hypercube.

Let $A_1,\ldots, A_{2n}$ be the vertices of the cross-polytope
with $A_i$ be the all-zero vector, with one in the $i-$th position
for $i=1,\ldots, n$ and $A_i=-A_{i-n}$ for $i>n$. Also let
$\Gamma$ be a sphere, concentric to the sphere circumscribed
around the cross-polytope.

We are interested in the following question: For which points
$M\in\Gamma$ are the extremal values of
\begin{equation}\label{crossp}\sum_{i=1}^{2n}|MA_i|^{\lambda}=C_n(M,\lambda)\end{equation}
are achieved.

This question has been previously considered in \cite{Pacific2},
where full characterization of \eqref{crossp} when $\Gamma$ is the
sphere circumscribed around the cross-polytope and
$\lambda\in(0;2)$.

Put $B_i=OA_i\bigcap\Gamma$.

We shall prove the following

\begin{theorem}
\label{teoremacp}
\begin{enumerate}
\item $\lambda<0$. The maximum of $C_n(M,\lambda)$ is obtained
when $M\equiv B_i$ for some $i$. In the case when $\Gamma$ is the
sphere circumscribed around $A_1\ldots A_{2n}$ the sum
$C_n(M,\lambda)$ is not bounded. The minimum is obtained for some
vector $B=(\sum_{i=1}^n \pm B_i)/\sqrt N$ for any choice of plus
and minus signs. Those are the points obtained by the intersection
of the perpendicular from $O$ to some face of the cross-polytope
and $\Gamma$. \item $\lambda\in[0;6]$. If $\lambda$ is an even
number then $C_n(M,\lambda)$ is independent of the position of $M$
on $\Gamma$. Otherwise let $2m<\lambda<2m+2$. If $m$ is an even
(odd) integer, then the maximum (minimum) of $C_n(M,\lambda)$ is
obtained when $M$ coincides with the any of the points $B$,
defined as above, and the minimum (maximum) is achieved when
$M\equiv B_i$ for some $i$. \item $\lambda>6$ Then the maximum of
$C_n(M,\lambda)$ is obtained when $M$ coincides with any of the
vertices $B_i$, defined the minimum is obtained when $M\equiv B$
where $B$ is defined as above.
\end{enumerate}
\end{theorem}
\begin{proof}

As in the case of the regular simplex we shall actually consider
the more general function
\begin{equation}\label{crossp2}C_n(M,\lambda)=\sum_{i=1}^{2n}
(|MA_i|^2+h)^{\lambda/2},\end{equation}
where $h$ is some fixed positive real number and prove that the
above theorem hods. We proceed by induction on the dimension of
the polytope.

We have already considered the desired result in the planar case.
We shall prove that for $\lambda=0,2,4,6$ the sum eq\ref{crossp2}
is independent on the position of $M$ on $\Gamma$ and that these
are all the powers with that property. We have already verified
this in the planar case and found it to be true. Now assume that
it is true for all cross-polytopes of dimension smaller than $n$.

Now consider the $n$-dimensional cross-polytope. Consider a
hyperplane $\pi_i$, perpendicular to $A_iA_{n+i}$, which
intersects $\Gamma$ such that the intersection is a sphere
$\Gamma'$ of dimension $n-2$. We shall prove that \eqref{crossp2}
is independent on the position of $M\in\Gamma'$ for
$\lambda=0,2,4,6$ and these are the only powers with that
property. We have that for $M\in\Gamma'$ $|MA_i|$ and $|MA_{n+i}|$
is constant, so we need only consider the function
\begin{equation}\sum_{i=j,j\neq i, n+i}^{2n}(|MA_j|^2+h)^{\lambda/2}=
\sum_{j=1,i\neq
i,n+i}^{2n}(|M'A_j|^2+H^2+h)^{\lambda/2}\label{crossp3},\end{equation}
where $M'$ is the projection of $M$ in the hyperplane spanned by
the vectors $A_i$, $i\neq i,n+i$. Now the projection of $\Gamma'$
in that hyperplane is a sphere circumscribed around the
$n-1$-dimensional cross-polytope with vertices $A_i$, $i\neq
i,n+i$, hence by the induction hypothesis \eqref{crossp3} is
independent o the position of $M$ on $\Gamma'$ for
$\lambda=0,2,4,6$ and those are the only such powers. Hence we
have proved the desired result for $M\in\Gamma'$, now considering
another plane $\pi_j$ again intersecting $\Gamma$ in an
$(n-2)$-dimensional sphere $\Gamma'_j$ and perpendicular to
$A_jA_{n+j}$. Again with the same argument we obtain that the sum
\eqref{crossp2} is independent on the position of $M$ on
$\Gamma'_j$ for $\lambda=0,2,4,6$ and those are the only powers
with this property.. Now using Lemma \ref{hcon} it follows that
this result can be extended to the whole sphere $\Gamma$.

We now proceed with the consideration of the extremal points. We
shall only consider the case or $\lambda<0$ as the desired results
in other cases of Theorem~\ref{teoremacp} can be obtained in the
same manner.

We shall apply the same approach we used with the regular simplex.
We again use induction on the dimension of the cross-polytope. Now
let $\lambda<0$. Assume that the minimum of \eqref{crossp} is
obtained at some point $N$ on $\Gamma$. Take a hyperplane $\pi_i$,
containing $N$ and parallel to the hyperplane $\sigma_i$ spanned
by the vectors $A_j$, $i\neq j, j+n$. Put
$\pi_i\bigcap\Gamma=\Gamma_i$. We have that $\dim\Gamma_i$ is
either zero (in that case $N$ coincides with some vertex of the
cross-polytope) or $\dim\Gamma_i=n-2$. There is at most one
non-parallel hyperplane for which the first case occurs. Take $i$,
such that $\dim\Gamma_i=n-2$. We have that $|MA_i|$ and
$|MA_{i+n}|$ are constant for $M\in\Gamma_i$, so we need only
consider the sum $\sum_{j=1, j\neq i,i+n}^{2n}|MA_j|^{\lambda}$.

Now we consider \begin{equation}\label{eq}\sum_{j=1, j\neq
i,i+n}^{2n} (|M'A_j|^2+H+h)^{\lambda/2},\end{equation} where $H$
is the squared distances between the planes $\pi_i$ and $\sigma_i$
and $M'$ is the projection of $M$ in the plane $\sigma_i$. We have
now that $M'$ belongs to the projection of $\Gamma_i$ in the plane
$\sigma_i$, which is a $(n-2)$-dimensional sphere, concentric to
the sphere circumscribed around the $(n-1)$-dimensional
cross-polytope with vertices $A_j$, $j\neq i,i+n$.

Now as the global minimum of $C_n(M,\lambda)$ is obtained at $N$,
so a local minimum must be obtained at that point, hence $N'$ is
the point for which \eqref{eq} achieves its minimum.

We consider the $n-1$-dimensional subspace, the hyperplane
$\sigma_i$. From the induction hypothesis and the fact that $N'$
is the point for which the sum \eqref{eq} achieves its minimum it
follows that the coordinates of $N'$  in that space are given by
$c\mathbf{v}$, where $c$ is a constant and $\mathbf{v}$ is the
vector, each coordinate of which is $\pm1$. Hence the $j$-th,
$j\neq i$ coordinate of $N$ is $\pm c$ and the $i$-th coordinate
of $N$ is $\pm\sqrt{H}$. Now consider another hyperplane $\pi_j$,
$j\neq i$ defined in the same way as $\pi_i$ (as $n\geq3$ such
plane exists), with analogous considerations it follows that the
$k-$th, $k\neq j$ coordinate of the point $N$ is $\pm s$, where
$s$ is a constant. From here it follows that the coordinates of
$N$ are all $\pm a$, where $a$ is a constant. Taking into
consideration that $N$ lies on the sphere $\Gamma$, the desired
result follows.

Now assume that $\Gamma$ is not the sphere circumscribed around
$A_1\ldots A_{2n}$, as in that case \eqref{crossp2} is not bounded
when $M\to A_i$ for some $i$. Assume that the maximum of
\eqref{crossp2} is achieved at a point $T$. As before we take a
hyperplane $\pi_i$, containing $T$ and parallel to the hyperplane
$\sigma_i$ spanned by the vectors $A_j$, $i\neq j, j+n$. Put
$\pi_i\bigcap\Gamma=\Gamma_i$. We have that $\dim\Gamma_i$ is
either zero-in that case $N$ coincides with some vertex of the
cross-polytope and we have nothing to prove. Assume that this is
not the case, then $\dim\Gamma_i=n-2$. We have that $|MA_i|$ and
$|MA_{i+n}|$ are constant for $M\in\Gamma_i$, so we need only
consider the sum $\sum_{j=1, j\neq i,i+n}^{2n}|MA_j|^{\lambda}$.

As before consider the projection of the point $M$ in the
hyperplane $\sigma_i$, spanned by the vectors $A_j$, $j\neq
i,i+n$. Then again by the induction hypothesis it follows that if
$T$ projects in $T'$, then in the subspace $\sigma_i$ $T'$ has at
most one nonzero coordinate, let it be the $k$-th coordinate. It
follows that $T$ has at most two nonzero coordinates. As $T$
belongs to $\Gamma$, $T$ cannot coincide with the origin of the
coordinate space. If it has one nonzero coordinate, the conclusion
follows. Now assume that it has two nonzero coordinates, say $k$
and $i$. Take now another hyperplane $\pi_j$, $j\neq k,i$, defined
in the same fashion as $\pi_i$, as $n\geq3$ such a hyperplane
exists. Applying the same methodology it follows that at it is not
possible both the $k$-th and the $i$-th coordinate to be nonzero,
hence $T$ has exactly one nonzero coordinate. Taking into
consideration that $|T|$ is equal to the radius of $\Gamma$ the
desired result follows.
\end{proof}

We may consider the following

\begin{conjecture}
Let $A_1,\ldots,A_{2n}$ be different points in  $\mathbb{R}^n$.
Assume that there is an sphere $\Gamma$ with $\dim\Gamma=n-1$,
such that $\sum_{i=1}^{n+1}|MA_i|^{\lambda}$ is independent of the
position of $M$ on $\Gamma$ for $\lambda=2,4,6$, then $A_1,\ldots,
A_{2n}$ are the vertices of an $n$-dimensional cross-polytope.
\end{conjecture}
While this is true for the planar case, even the three-dimensional case is
very difficult and the authors do not have a proof for $n>2$.

\subsection{Hypercube}
We now begin with the consideration of the hypercube. For ease of
the introduction we shall consider the standard unit cube under
the translation with the vector $\mathbf{v}=(-1/2,\ldots,-1/2)$,
hence all the vertices of the cube have coordinates
$(\pm1/2,\ldots, \pm1/2)$.

We again investigate the extremal properties of the sum
\begin{equation}\label{cube}H_n(M,\lambda)=\sum_{i=1}^{2^n}(|MA_i|^2+h)^{\lambda/2},
\end{equation}
where $A_i$ are the vertices of the $n$-dimensional hypercube in
$\mathbb{R}^n$, and $M$ is a point on a sphere $\Gamma$,
concentric to the sphere circumscribed around the cube.

Previous work on this topic includes \cite{Pacific2}, where
partial results have been obtained f or the case in which $\Gamma$
is the sphere circumscribed around the polytope and
$\lambda\in(0;2)$.

Let $\mathbf{e}_i$ be the standard orthonormal vector base for $\mathbb{R}^n$,
and $r$ the radius of $\Gamma$.
Put $B_iOA_i\bigcap\Gamma$, where $O$ is the center of $\Gamma$.

We shall prove the following theorem
\begin{theorem}
\label{teoremahc}
\begin{enumerate}
\item $\lambda<0$ The maximum of $H_n(M,\lambda)$ is obtained when
$M=B_i$ for some $i$. In the case when $\Gamma$ is the sphere
circumscribed around $A_1\ldots A_{2^n}$ the sum $H_n(M,\lambda)$
is not bounded. The minimum is obtained for the point
corresponding to the vector $r\mathbf{e}_i$ for some $i$. Those
are the points obtained by the intersection of the perpendicular
from $O$ to some $n-1$face of the cube and $\Gamma$. \item
$\lambda\in[0;6]$. If $\lambda$ is an even number, then
$H_n(M,\lambda)$ is independent of the position of $M$ on
$\Gamma$. Otherwise let $2m<\lambda<2m+2$. If $m$ is an even (odd)
integer, then the maximum (minimum) of $H_n(M,\lambda)$ is
obtained when $M$ is the point corresponding to some vector
$r\mathbf{e}_i$,. The minimum(maximum) is achieved when $M\equiv
B_i$ for some $i$. \item $\lambda>6$ Then the maximum of
$H_n(M,\lambda)$ is obtained when $M$ coincides with a vertex
$B_i$, and the minimum is obtained when $M$ is the point
corresponding to some vector $r\mathbf{e}_i$.
\end{enumerate}
\end{theorem}
The points $B_i$ can also be characterized by the vectors
$\frac{r}{\sqrt{n}}(\pm1/2,\ldots,\pm1/2)$, and this
characterization we shall use in the proof.

The results for the hypercube are very similar to those obtained
for the cross-polytope. One may attribute this to the polytopes
being dual. It is worth exploring similarities in the properties
of functions of the type of \eqref{cube} for dual polytopes.

\begin{proof}
We shall first prove that \eqref{cube} is independent on the
position of $M$ on $\Gamma$ for $\lambda=0,2,4,6$, and that these
are the only powers for which this is true.

We shall proceed as before with induction on the dimension of the
cube. The result has already been proved for the two-dimensional
case. Assume it to be true for all dimensions less than $n$.

Let $\pi_i$ and $\pi_i'$ be two planes parallel to the hyperplane
defined by $x_i=0$ and containing all of the vertices of the cube
(each hyperplane contains one $n-1$ face of the cube). Let now
$\sigma_i$, be a hyperplane, parallel to $\pi_i$, such that
$\dim\sigma_i\bigcap\Gamma=\Gamma_i=n-2$. We shall prove that
\eqref{cube} is independent on the position of $M$ on $\Gamma_i$
for $\lambda=0,2,4,6$.

Let the $n-1$ face of the hypercube, contained in $\pi_i$ be $S$
and the $n-1$ face contained in $\pi_i'$ be $S'$. The vertices of
each of $S$ and $S'$ are themselves vertices of $n$-dimensional
cubes. Now we have

\begin{eqnarray}\nonumber H_n(M,\lambda)|_{M\in\Gamma_i}=\sum_{i=1}^{2^n}
(|MA_i|^2+h)^{\lambda/2}=\sum_{A_i\in
S}(|M_1A_i|^2+H_1+h)^{\lambda/2}+ \\\sum_{A_i\in
S'}(|M_2A_i|^2+H_2+h)^{\lambda/2}, \end{eqnarray} where $M_1$ is
the projection of the point $M$ in the plane $\pi_i$, $H_1$ is the
squared distance between the planes $\pi_i$ and $\sigma_i$, $M_2$
is the projection of $M$ in $\pi_i'$, and $H_2$ is the squared
distance between the planes $\pi_i'$ and $\sigma_i$. Now as the as
the vertices $A_i\in S$ and $A_j\in S'$ are vertices of lower
dimensional cubes $H_1$ and $H_2$, and $\Gamma_i$ projects in each
of the planes $\pi_i$ and $\pi_i'$ to spheres concentric to the
spheres circumscribed around $H_1$ and $H_2$ from the inductive
hypothesis it follows that each of the sums $\sum_{A_i\in
S}(|M_1A_i|^2+H_1+h)^{\lambda/2}$, $\sum_{A_i\in
S'}(|M_2A_i|^2+H_2+h)^{\lambda/2}$ is independent of the position
of $M$ in $\Gamma_i$ for $\lambda=0,2,4,6$. Taking another two
planes $\pi_j$ and $\pi_j'$, $i\neq j$, defined in the same
fashion as $\pi_i$ and $\pi_i'$ and with the same considerations
it follows that \ref{cube} is also independent of the position of
$M$ on $\Gamma_j$ (defined in the same way as $\Gamma_i$) for
$\lambda=0,2,4,6$. Now by Lemma \ref{hcon} it follows that this
can be extended to the whole sphere $\Gamma$. To prove that
$\lambda=0,2,4,6$ are the only powers with that property we need
only consider a plane $\sigma_i$, defined by $x_i=0$, then taking
$\pi_i$ and $\pi_i'$ as before, we have
\begin{eqnarray}\nonumber H_n(M,\lambda)|_{M\in\Gamma_i}=
\sum_{i=1}^{2^n}(|MA_i|^2+h)^{\lambda/2}
=\sum_{A_i\in S}(|M_1A_i|^2+H_1+h)^{\lambda/2}+ \\
\label{cubeplanes}+\sum_{A_i\in
S'}(|M_2A_i|^2+H_2+h)^{\lambda/2}=2\sum_{A_i\in
S}(|M_1A_i|^2+H_1+h)^{\lambda/2},\end{eqnarray} due to symmetry
and the desired result follows immediately from the induction
hypothesis, as $M_1$ belongs to sphere concentric to the sphere
circumscribed around the $n-1$-dimensional cube $H_1$.

Now we shall prove the extremal cases of Theorem~\ref{teoremahc}.

We will only consider the case $\lambda<0$ as other cases can be
proved in the same way. We proceed by induction. The have already
proved the claim for planar case.

Assume that Theorem~\ref{teoremahc} is true and assume that the
maximum of \ref{cube} is achieved at some point $N$. We shall
prove that $N=(\pm c,\ldots, \pm c)$. Consider a plane $\sigma_i$,
parallel to the plane $x_i=0$ and as above consider the planes
$\pi_i$ and $\pi_i'$. Suppose that
$\dim\sigma_i\bigcap\Gamma=\Gamma_i=n-2$, this is possible as for
every point $N$ there is exactly one class of nonparallel
hyperplanes such that $\dim\sigma_i\bigcap\Gamma=\Gamma_i=0$ and
in the case we consider $n\geq3$. Now again we consider
$H_n(M,\lambda)|_{\Gamma_i}$. As the global maximum of
\eqref{cube} is achieved in $N$, so it is the local, hence $\max
H_n(M,\lambda)|_{\Gamma_i}=H_n(N,\lambda)|_{\Gamma_i}$ for a fixed
$\lambda<0$.

We again have the identity \eqref{cubeplanes} holds, and due to
symmetry and the induction hypothesis it follows that $N_1$ (the
projection of $N$ in $\pi_i$) is the point for which the sum
$\sum_{A_i\in S}(|M_1A_i|^2+H_1+h)^{\lambda/2}$ is maximized and
$N_2$ (the projection of $N$ in $\pi_i'$) the point for which
$\sum_{A_i\in S'}(|M_2A_i|^2+H_2+h)^{\lambda/2}$ is maximized.
Hence again by the induction hypothesis it follows that the all
the coordinates of $N$, except the $i-th$ are of the form $\pm
c_0$. Taking another plane $\sigma_j$, $j\neq i$, defined in the
same fashion (this is possible as $n\geq3$) and repeating the
above reasoning we obtain that all the coordinates of $N$, except
the $j-$th are of the form $\pm c_1$. Now as $n\geq3$ it follows
that all coordinates of $N$ are $\pm c$, taking into account that
$|ON|=r$ the conclusion follows.

Assume that the minimum of \eqref{cube} is obtained for a point
$T$. We now begin with the proof of the minimal case. Again we
consider a plane $\sigma_i$, parallel to the plane $x_i=0$
containing $T$. Assume that
$\dim\sigma_i\bigcap\Gamma=\Gamma_i=0$, and in that case we have
the desired result $N=\mathbf{e}_i$. Assume now that
$\dim\sigma_i\bigcap\Gamma=\Gamma_i=n-2$.

Now we again consider $H_n(M,\lambda)|_{\Gamma_i}$, for which
\eqref{cubeplanes} holds, and again by symmetry and the induction
hypothesis we have that $T_1$ (the projection of $T$ in $\pi_i$)
is the point for which the sum $\sum_{A_i\in
S}(|M_1A_i|^2+H_1+h)^{\lambda/2}$ is minimized and $T_2$ (the
projection of $T$ in $\pi_i'$) the point for which  $\sum_{A_i\in
S'}(|M_2A_i|^2+H_2+h)^{\lambda/2}$ is minimized. It follows that
among all the coordinates of $N$, except the $i-$th there is at
most one nonzero, say the $j-$th. Now if we again consider a
hyperplane $\sigma_k$, $k\neq i,j$ defined in the same fashion as
$\sigma_i$ (this is possible due to $n\geq3$ and the consideration
of $\dim\sigma_i\bigcap\Gamma=\Gamma_i=0$ done in the beginning)
and repeating the above considerations it follows that among all
coordinates of $T$, except the $k-$th there is at most one
nonzero. Now we easily get that $N$ has at most one nonzero
coordinate. As we have that $|OT|=r$ we obtain the desired result.
\end{proof}

Another question which we can consider is the following: Let
$A_1,\ldots,A_{2^n}$ be different points in  $\mathbb{R}^n$.
Assume that there is a sphere $\Gamma$ with $\dim\Gamma=n-1$, such
that $\sum_{i=1}^{2^n}|MA_i|^{\lambda}$ is independent of the
position of $M$ on $\Gamma$ for $\lambda=2,4,6$, then $A_1\ldots
A_{2^n}$ are the vertices of an $n$-dimensional cube.

But this is not generally true. In $\mathbb{R}^n$, $n$ being a
power of two greater than first, we can consider $2^n/2n$
different rotations under which the images of no two points
coincide of the cross-polytope. The set of points obtained in this
way satisfies the condition. No counterexamples are known to the
authors in different dimensions.

One can see many similarities between the cases with the
cross-polytope and the cube. We have proved the following

Take the standard cross-polytope $A_1\ldots A_{2n}$ and the unit
cube described in the beginning of this section $B_1\ldots
B_{2^n}$, and consider each of the sums \eqref{crossp} and
\eqref{cube}. Assume that for some $\lambda$ the minimum of
$\eqref{crossp}$ is obtained at some point $N$, then the maximum
of \eqref{cube} is obtained at $N$ and vice versa. This may be to
the fact of the duality of the two polytopes. Here based on the
results in the planar case, the results for the regular simplex,
and the dual pair cube, cross-polytope for which it holds, we  can
formulate the following:

\begin{conjecture}
Let $P$ be an $n$-dimensional regular polytope, in $\mathbb{R}^n$,
inscribed in the unit sphere $U$. Let $S$ be the polytope obtained
under the polar reciprocation of $P$ in $U$. Put $P'$ be a
homothetic polytope of $S$, inscribed in $U$. Let $\mathbf{V}$ be
the set of vertices of $P$, and $\mathbf{V'}$ the set of vertices
of $P'$. Consider the sums
\begin{equation}\label{1}\sum_{A\in\mathbf{V}}|MA|^{\lambda},\end{equation}
\begin{equation}\label{2}\sum_{A\in\mathbf{V'}}|MA|^{\lambda},\end{equation}
where $M$ belongs to some sphere concentric with the units sphere. Assume that

the maximum of \eqref{1} is obtained for some point $N$, then the
minimum of \eqref{2} is obtained at $N$, and vice versa. If
\eqref{1} is independent of the movement of $M$ on $\Gamma$ then
so is \eqref{2} and vice versa.
\end{conjecture}
Trough the results proved so far we have verified this proposition
for the regular polygons,  regular simplex, and the dual pair
cube, cross-polytope. It remains to be proved for the dual pairs
icosahedron and the dodecahedron, the $600$-cell and the $120$-
cell in $4$-dimensional space.
\medskip

{\bf Acknowledgement.} The authors would like to thank the referee
for his comments and remarks, which helped improve the clarity of
the paper.

\end{document}